\documentclass[11pt]{amsart}
\usepackage[utf8]{inputenc}
\usepackage{amssymb}
\usepackage{amsthm}
\usepackage{amsmath}
\usepackage{subcaption}
\usepackage{theoremref}
\usepackage{mathrsfs} 
 \usepackage{enumitem}
 \usepackage{xcolor}

\newtheorem{theorem}{Theorem}[section] 

\newtheorem{lemma}[theorem] {Lemma} 
 
\theoremstyle{definition}
\newtheorem{definition}[theorem]{Definition}
\newtheorem{example}[theorem]{Example}

\newcommand\G{\mathfrak{G}}
\newcommand\fH{\mathfrak{H}}
\newcommand\bbZ{\mathbb{Z}}
\newcommand\pdot{{\cdot}}
\newcommand\supp{\operatorname{supp}}
\newcommand\textb{\textrm{b}}
\newcommand\Lat{\operatorname{Lat}}

\begin{document}

\title{The Rhodes semilattice of a biased graph }
\author{Michael J.\ Gottstein and Thomas Zaslavsky}
\address{Department of Mathematics and Statistics \\ Binghamton University (SUNY) \\ Binghamton, New York 13902-6000 USA}
\email{gottstein@math.binghamton.edu, zaslav@math.binghamton.edu}

\begin{abstract}
    
We reinterpret the Rhodes semilattices $R_n(\G)$ of a group $\G$ in terms of gain graphs and generalize them to all gain graphs, both as sets of partition-potential pairs and as sets of subgraphs, and for the latter, further to biased graphs. Based on this we propose four different natural lattices in which the Rhodes semilattices and its generalizations are order ideals.

\end{abstract}



\maketitle

\section{Introduction}\label{intro}

The Rhodes semilattice $R_n(\G)$ of a group $\G$, introduced by John Rhodes for semigroup theory (in \cite{R, AHNR}; we particularly refer to \cite{Rhodes}), is a partial ordering of pairs consisting of a partition and a potential system on subsets of a finite set.  We reinterpret the Rhodes semilattice as a semilattice of subgraphs in a complete link gain graph, which is a graph with edges labeled from the group $\G$.  From this point of view the elements of the Rhodes semilattice are closed, balanced subgraphs and the ordering is by gain-graph inclusion.  We use this reinterpretation to generalize Rhodes semilattices to semilattices of closed, balanced subgraphs of any gain graph and even more generally any biased graph.  (All these concepts are defined below.)  Based on our graphic interpretation we propose three natural lattices that contain as order ideals the Rhodes semilattice and its generalizations.  

The graphical interpretation of $R_n(\G)$ was initiated by the second author in a referee's report on \cite{Rhodes} and is mentioned in the published version; see \cite[Remark 6.1]{Rhodes}.  Here we provide the full generalization suggested by that insight, with proof.  Our partition-potential generalization of $R_n(\G)$ is newly developed by the first author.

\section{A new perspective on the Rhodes semilattice}\label{newrhodes}

We begin with basic definitions, many of which are from \cite{Rhodes} and \cite{Zas1}.

A \emph{partial partition} $\pi=\{\pi_1, \ldots, \pi_j\}$ of a set $X$ is a partition of a subset of $X$.  The \emph{support} $\supp\pi$ is the union of all blocks of $\pi$. 
The set of partial partitions of $X$ is partially ordered by refinement, i.e., $\tau \leq \pi$ if every block of $\tau$ is contained in a block of $\pi$.  The smallest partial partition of $X$ is the empty partition, $\emptyset$; the largest is the trivial partition, $\{X\}$.  

For a graph $\Upsilon$ we define the partition of $V(\Upsilon)$ induced by $\Upsilon$ by $\pi(\Upsilon)=\{V(D): D \text{ is a component of } \Upsilon \}$. 

Let $X$ be a finite set of order $n$ and $\G$ a group. Given $I \subseteq X$, we denote by $\G^I$ the collection of all functions $\theta: I \rightarrow \G$; we call such a function a \emph{potential function} on $I$. The group acts on the left of $\G^I$ by $(g \theta)(x)=g \theta(x)$; we call an element of the quotient set $\G^I / \G$ a \emph{potential} for $I$ (\cite{Rhodes} calls it a ``cross section''). 

A \emph{potential system} for a partial partition $\pi$ is a collection of potentials on the blocks of $\pi$.  Formally, if $\theta: Y \to \G$ and $\eta: Z \to \G$, where $Y$ and $Z$ are any subsets of $X$ that contain $\supp\pi$, we write $\eta \sim_\pi \theta$ if $\eta|_{\pi_i} \in \G \cdot \theta|_{\pi_i}$ for each block $\pi_i$ in $\pi$; then $\sim_\pi$ is an equivalence relation.  We denote the equivalence class of $\theta \in \G^I$ for a partial partition $\pi$ by $[\theta]_\pi$; this is a \emph{potential system} on $\pi$.  

The \emph{Rhodes semilattice} (see \cite[section 3]{Rhodes}) is the set of all pairs $(\pi, [\theta]_{\pi})$ in which $[\theta]_{\pi}$ is a potential system on $\pi$, which we call \emph{partition-potential pairs}.  They are partially ordered by restriction; that is, $(\tau, [\eta]_{\tau}) \leq (\pi, [\theta]_{\pi}])$ if $\tau$ refines $\pi$ and $[\eta]_\tau = [\theta]_\tau$.  
The Rhodes semilattice has a meet operation defined by $(\pi,[\theta]_{\pi}) \wedge (\tau,[\eta]_{\tau}) = (\rho,[\theta]_{\rho}) = (\rho,[\eta]_{\rho})$, where $\rho$ is the partition of $\supp\rho = \supp\pi \cap \supp\tau$ induced by the equivalence relation $\sim$ defined on $\supp\rho$ in the following way: $x \sim y$ if and only if there exist blocks $\tau_i$ and $\pi_j$ such that $x,y \in  \tau_i \cap \pi_j$ and $\theta(x)^{-1}\theta(y)=\eta(x)^{-1}\eta(y)$.  To distinguish this representation of the Rhodes semilattice from the graphical one to be defined, we call it the \emph{partition-potential Rhodes semilattice}.

In what follows $\Gamma=(X,E)$ will denote a graph, which may have loops and multiple edges unless it is said to be simple. If $e\in E$, $\nu_{\Gamma}(e)=\{v, w\}$ means $e$ has endpoints $v,w$ (a multiset to account for loops). If $\Upsilon \subseteq \Gamma$ is a subgraph, $V(\Upsilon)$ is its vertex set and $E(\Upsilon)$ is its edge set.  We will be most interested in the subgraph induced on a subset $I$ of $X$, which is $\Gamma{:}I=(I, E{:}I)$ where
$$
E{:}I=\{e \in E: \nu_{\Gamma}(e) \subseteq I \},
$$
and the subgraph induced by a partial partition $\pi$ of $X$, which is
$$
\Gamma{:}\pi=\bigcup\{(\Gamma{:}\pi_i): \pi_i \in \pi\} .
$$

A \emph{gain graph}, denoted by $\Phi=(\Gamma,\phi,\G)$, is a graph equipped with a \emph{gain function} $\phi$. The gain function is defined on the edges of the graph to the group $\G$, such that reversing the direction of an edge inverts the gain. We write $\phi(e ; v, w)$ in order to indicate the sense in which the gain is measured; thus $\phi(e ; w, v)=\phi(e ; v, w)^{-1}$. The gain of a path is its edge gain product; thus a path $P=v_0e_1 v_1e_2 \cdots e_kv_k$ has gain $\phi(P)=\phi(e_1;v_0,v_1) \phi(e_2;v_1,v_2) \cdots \phi(e_k;v_{k-1},v_k).$ A circle is called \emph{balanced} if its gain (considered as a closed path) is $1$; this property is independent of its representation as a closed path.  The set of balanced circles of $\Phi$ is denoted by $\mathscr{B}(\Phi)$.  A subgraph $\Upsilon$ of a gain graph $\Phi$ is balanced if every circle in $\Upsilon$ is balanced. It is \emph{closed and balanced} if it is balanced and whenever there is a balanced circle $C\subseteq \Phi$ such that $C\setminus e \subseteq \Upsilon $, then $e \in E(\Upsilon)$.

\begin{lemma}\thlabel{L1} 
In a balanced gain graph, the gain of a path depends only on its initial and final vertices. 
\end{lemma}
\begin{proof} 
This is implicit in the proof of \cite[Lemma 5.3]{Zas1}.
\end{proof}

A \emph{potential function} for a balanced gain graph $\Phi$ is a function $\theta: V(\Phi) \rightarrow \mathfrak{G}$ such that $\phi(e ; v, w)= \theta(v)^{-1} \theta(w)$ for each edge $e \in E(\Gamma) $.

\begin{lemma}\thlabel{L2} 
Let $\Phi$ be a gain graph. There is a potential function $\theta$ for $\Phi$ if and only if $\Phi$ is balanced.  Every potential function in $[\theta]_{\pi(\Phi)}$ defines the same gains on $\Phi$, and every potential function defining the gains of $\Phi$ is in $[\theta]_{\pi(\Phi)}$.
\end{lemma}

\begin{proof} 
Balance is clear by computing the gain of a circle.

Conversely, we can obtain a potential function for a balanced gain graph $\Phi$ by choosing a root node $r$ for each component of $\Phi$ and defining $\theta(v)=\phi(P_{r v})$ where $P_{r v}$ is an $r v$-path in a component of $\Phi$. This is well defined because the gain function in a balanced gain graph is path independent. 

Suppose $\phi(e)= \theta(v)^{-1} \theta(w)$ and $\eta$ is another potential function of $\Phi$, then (because the gain function in a balanced subgraph is path independent) for any path $P_{vw}$ in $\Phi$, $(\theta(v))^{-1}\theta(w)=\phi(P_{vw})=(\eta(v))^{-1}\eta(w)$. This is true for every $v,w$ in a block of $\pi(\Phi)$, so by definition $\eta \in [\theta]_{\pi(\Phi)}$. Conversely, if $\eta \in [\theta]_{\pi(\Phi)}$ then $(\theta(v))^{-1}\theta(w)=(\eta(v))^{-1}\eta(w)$ for any $v,w$ in a block of $\pi(\Phi)$, so $\phi(e)=(\theta(v))^{-1}\theta(w)=(\eta(v))^{-1}\eta(w)$.
\end{proof}

The \emph{group expansion} of a graph $\Gamma$ \cite[Example 6.7]{Zas1} is denoted by $\G \pdot \Gamma = (V(\Gamma), \G \times E(\Gamma),\phi)$, where $\nu_{\G \pdot \Gamma}(g,e) = \nu_\Gamma(e)$.  
To define the gain of $(g,e)$ we must take account of the endpoints, $\nu_{\G \pdot \Gamma}(g,e) = \{v,w\}$.  For the sake of notation, arbitrarily pick an orientation, $(e;v,w)$, and define $\phi(g,e;v,w)=g$, $\phi(g,e;w,v)=g^{-1}$ for each $g$.  In this notation, $(g,e)$ and $(g^{-1},e)$ are different edges (unless $g=g^{-1}$) whose gains are 
\begin{gather*}
\phi(g,e;v,w)=g,\ \phi(g,e;w,v)=g^{-1}, \\
\phi(g^{-1},e;v,w)=g^{-1},\ \phi(g^{-1},e;w,v)=g.
\end{gather*}

\begin{definition}
 The \emph{graphic Rhodes semilattice of $\G\pdot K_n$}, denoted by $R^\textb(\G\pdot K_n)$, is the family of closed and balanced subgraphs in $\G\pdot K_n$, ordered by inclusion.  (The superscript $\textb$ distinguishes this from the partition-potential Rhodes semilattice.)  Its meet operation is intersection.
\end{definition}

\begin{theorem}\label{Rhodesiso}
The partition-potential Rhodes semilattice $R_n(\G)$ is isomorphic to the graphic Rhodes semilattice $R^\textb(\G\pdot K_n)$.
\end{theorem}

\begin{proof}
In the natural correspondence between $R_n(\G)$ and $R^\textb(\G\pdot K_n)$, the pair $(\pi,[\theta]_{\pi})\in R_n(\G)$ corresponds to the subgraph of $\G\pdot K_n$ whose components are complete subgraphs in each vertex set $\pi_i$ with gains given by any potential function in $[\theta]_{\pi}$. We correspond a closed and balanced subgraph $B$ of $\G\pdot K_n$ to the pair $(\pi(B),[\theta]_{\pi(B)})$, where $\theta$ is a potential function defining the gains of $B$. 
In Theorem \ref{isomorphism} we prove (in more generality) that this correspondence is an isomorphism. 
\end{proof}

Theorem \ref{Rhodesiso} shows us how to generalize the Rhodes semilattice to gain and biased graphs, which we do in the next sections.  We note that it is equivalent to \cite[Theorem 4.1]{Rhodes}, which uses the language of groupoids, as we explain in Section \ref{groupoid}.  We have given Theorem 2.4 as preparation for the generalization in Section \ref{gain}.

\section{Generalization to gain graphs}\label{gain}

Let $\Phi$ be a gain graph with gain group $\G$ and vertex set $X$. Let $\pi=\{\pi_1,\ldots,\pi_j\}$ be a partial partition of $X$ and $[\theta]_{\pi}$ a potential system for $\pi$. We define a function $\mathbf{B}$ from partition-potential pairs to subgraphs of $\Phi$ by 
$$\mathbf{B}(\pi,[\theta]_{\pi}):= (\supp\pi, \{e \in E(\Phi{:}\pi) : (\exists i)\ \nu_{\Gamma}(e)=\{v,w\} \subseteq \pi_i,\ \phi(e) = \theta^{-1}(v)\theta(w)\}).$$  
This subgraph is well defined, by \thref{L2}.

 We say $(\pi,[\theta]_{\pi})$ is a \emph{$\Phi$-connected} partition-potential pair if $\mathbf{B}(\pi,[\theta]_{\pi_i})$ is connected for each $\pi_i \in \pi$.  

 \begin{definition} 
Let $\Phi$ be a gain graph with vertex set $X$ and group $\G$. The \emph{partition-potential Rhodes semilattice of $\Phi$}, denoted by $R(\Phi)$, is the set of all $\Phi$-connected partition-potential pairs of $\Phi$. The meet operation is the same as it is for $R_n(\G)$. 
\end{definition}

We can see that $R_n(\G) = R(\G\pdot K_n)$, so the partition-potential Rhodes semilattice of a gain graph is a generalization of the original Rhodes semilattice.

\begin{definition} 
Let $\Phi$ be a gain graph. The \emph{graphic Rhodes semilattice of $\Phi$}, denoted by $R^\textb(\Phi)$, is the family of closed and balanced subgraphs, ordered by inclusion. The meet operation is intersection.
\end{definition}

\begin{lemma}\thlabel{L3}
Let $\Phi$ be a gain graph, $\pi$ a partial partition of $X$, and $(\pi,[\theta]_{\pi})\in R(\Phi)$.  Then $\mathbf{B}(\pi,[\theta]_{\pi})$ is a closed and balanced subgraph of $\Phi$.
\end{lemma}

\begin{proof}
$\mathbf{B}(\pi,[\theta]_{\pi})$ is balanced because it has gains defined by a potential function. It is closed because if $v,w$ are in a component of $\mathbf{B}(\pi,[\theta]_{\pi})$, then the edge $e$ with $\nu_{\Gamma}(\theta^{-1}(v)\theta(w),e)=\{v,w\}$ and gain  $\phi(e;v,w) = \theta^{-1}(v)\theta(w)$, if it exists in $\Phi$, is in $\mathbf{B}(\pi,[\theta]_{\pi})$ by the definition of $\mathbf{B}$.
\end{proof}

 \begin{lemma}\thlabel{L3.5}
Let $\Phi$ be a gain graph.  $\mathbf{B}$ is a surjection onto the closed and balanced subgraphs of $\Phi$.
\end{lemma}

\begin{proof}
By \thref{L3} we know the image of $\mathbf{B}$ is contained in the set of closed and balanced subgraphs.

Let $B$ be a closed and balanced subgraph.  By \thref{L2} there is a potential function $\theta$ defining its gains. If $e\in B$ then its vertices are in the same block of $\pi(B)$; it follows that $e$ is in $\mathbf{B}(\pi(B),[\theta]_{\pi(B)})$, which implies that $B\subseteq \mathbf{B}(\pi(B),[\theta]_{\pi(B)})$. 

Now let $e$ be an edge of $\mathbf{B}(\pi(B),[\theta]_{\pi(B)})$.  The vertices of $e$ are in one block of $\pi(B)$ and its gain $\phi(e;v,w) = \theta(v)^{-1}\theta(w)$.  Since $B$ is closed and $\theta$ is a potential function for $B$, $e$ is in $B$; therefore $B=\mathbf{B}(\pi(B),[\theta]_{\pi(B)})$.
\end{proof} 

 \begin{lemma}\thlabel{L4}
Let $\Phi$ be a gain graph. If we restrict the domain of $\mathbf{B}$ to $R(\Phi)$, then $\mathbf{B}$ is injective.
 \end{lemma}

\begin{proof}
Suppose $(\tau,[\eta]_{\tau}) \in R(\Phi)$ and $\mathbf{B}(\tau,[\eta]) = B$.  We showed in the previous proof that $B=\mathbf{B}(\pi(B),[\theta]_{\pi(B)})$.  We now prove that, if $(\tau,[\eta]_{\tau}) \neq (\pi(B),[\theta]_{\pi(B)})$, then $(\tau,[\eta]_{\tau}) \notin R(\Phi)$.

First we observe that every block of $\pi(B)$ must be contained in a block of $\tau$.  Thus, $\pi(B) \leq \tau$.  If two blocks $\pi_i, \pi_j \in \pi(B)$ are contained in the same block $\tau_k$, or if some $\pi_i \subset \tau_k$, then $\tau_k$ does not induce a connected subgraph of $B$; such a partition-potential pair cannot be in $R(\Phi)$.  Thus, the blocks of $\pi(B)$ are in distinct blocks of $\tau$ and are equal to those blocks; that is, $\tau = \pi(B)$.

Now \thref{L2} implies that $[\theta]_\tau = [\theta]_{\pi(B)}$, completing the proof.
\end{proof}

\begin{theorem}\label{isomorphism}
Let $\Phi$ be a gain graph. Then $R(\Phi)$ is isomorphic to $R^\textb(\Phi)$.
\end{theorem}

\begin{proof}
 We have shown $\mathbf{B}$ is a bijection. Now we prove it preserves order. By the definition $(\tau,[\eta]_\tau)\leq(\pi,[\theta]_\pi)$ if and only if  each block $\tau_i \in \tau$ is contained in some block of $\pi$ and $\G\pdot \theta|_{\tau_i}= \G\pdot \eta|_{\tau_i}$.  Equivalently, each component of $\mathbf{B}(\tau,[\eta]_{\tau})$ is contained in a component of $\mathbf{B}(\pi,[\theta]_{\pi})$; that is,  $\mathbf{B}(\tau,[\eta]_{\tau})\subseteq \mathbf{B}(\pi,[\theta]_{\pi})$. 
\end{proof}

\begin{example}[Group expansions]
If $\Phi$ is a group expansion we can simplify the description of the graphic Rhodes semilattice. Suppose $\Gamma$ is a simple graph and $\G$ is a group.  A subgraph $\Psi \subseteq \Gamma$ is \emph{closed in $\Gamma$} if, whenever $e$ is an edge of $\Gamma$ such that $\Psi\cup \{e\}$ contains a circle that contains $e$, then $e$ is in $\Psi$.

For a subgraph $B$ of $\G\pdot\Gamma$, by $p(B)$ we mean the projection of $B$ onto the underlying graph $\Gamma$.  
If $B$ is a closed and balanced subgraph, then $p(B)$ is a closed subgraph in $\Gamma$.  Conversely, if $C$ is closed in $\Gamma$ and $[\theta]_{\pi(C)}$ is a potential system on $\pi(C)$, then $\mathbf{B}(\pi(C),[\theta]_{\pi(C)})$ is a closed and balanced subgraph of $\G\pdot\Gamma$.  

This shows that the properties of closure and balance for elements of $R^\textb(\G\pdot\Gamma)$ can be split into closure of the underlying base graph in $\Gamma$ and an arbitrary choice of potential for that base graph.  If in particular $\Gamma = K_n$, the closed subgraphs correspond bijectively to the partial partitions of the vertex set of $K_n$.
\end{example}

\section{Generalization to biased graphs}\label{bias}

A \emph{biased graph} is a graph together with a class of circles (edge sets or graphs of simple closed paths), called \emph{balanced circles}, such that no theta subgraph contains exactly two balanced circles. (A theta graph consists of three paths with the same two endpoints but otherwise disjoint from each other.) We denote the graph along with the set of balanced circles by $\Omega = (\Gamma, \mathscr{B})$. 
A subgraph $\Upsilon$ of a biased graph $\Omega$ is balanced if every circle in $\Upsilon$ is balanced and is \emph{closed and balanced} if in addition whenever there is a balanced circle $C\subseteq \Omega$ such that $C\setminus e \subseteq \Upsilon$, $C$ is in $\Upsilon$.  A gain graph $\Phi$ with underlying graph $\Gamma$ gives rise to the biased graph $\langle\Phi\rangle = (\Gamma, \mathscr{B}(\Phi))$ \cite[Section 5]{Zas1}.

The definition of the graphic Rhodes semilattice of a gain graph depends on the subgraphs and the balanced circles in the subgraph.  Since the balanced circles of a gain graph define a biased graph, we can readily generalize the definition of the graphic Rhodes semilattice to biased graphs. 

\begin{definition} \label{rhodesbg}
Let $\Omega$ be a biased graph. The \emph{Rhodes semilattice of $\Omega$}, denoted by $R^\textb(\Omega)$, is the family of closed and balanced subgraphs in $\Omega$ ordered by inclusion.
\end{definition}

We observe that the Rhodes semilattice is the disjoint union of the semilattices of balanced flats of the matroids of all induced subgraphs of $\Omega$, i.e., $\bigsqcup_{X \subseteq V(\Omega)} \Lat^\textb(\Omega{:}X)$.  The matroid can be either the frame or lift matroid, which we describe in Section \ref{lattice}. 

A difference between bias and gains is that we cannot state a partition-potential description of balanced subgraphs of a biased graph.  This is not a trivial difference, since not all biased graphs can be given gains; see \cite[Example 5.8]{Zas1}.

\section{Interlude on groupoids}\label{groupoid}

A \emph{groupoid} is a small category in which every morphism is an isomorphism, i.e., it has an inverse morphism.  We call the groupoid \emph{connected} if every pair of objects has a morphism; if it is not connected the objects fall into subsets that are connected.  Connected groupoids on the one hand, and group expansions of complete graphs on the other, are notational variants of each other, in which inverse pairs of morphisms correspond to edges of the gain graph.  Thus gain graphs can be regarded as a generalization of groupoids; in fact the definition of gain graphs by Doubilet, Rota, and Stanley \cite{DRS} is effectively a generalization of the definition of a groupoid, since it treats an edge as a pair of inversely labeled directed edges.

A \emph{trivial groupoid} is a groupoid in which there is at most one morphism from any object to any other object.  Trivial groupoids correspond to balanced gain graphs in which every component is complete; thus, the closed and balanced subgraphs in Theorem 2.4 are a notational variant of the trivial subgroupoids in \cite[Theorem 4.1]{Rhodes}, as is noted in \cite[Remark 6.1]{Rhodes}.

The principal difference between groupoids and group expansion graphs is that the latter are more general.  However, the generalization is not as great as may appear.  A group expansion $\G\cdot\Gamma$ of a connected incomplete graph $\Gamma$ extends uniquely to the group expansion $\G\cdot K_n$.  Nevertheless, gains open possibilities that will be explored elsewhere.

\section{Four lattices}\label{lattice}

In this section we have a biased graph $\Omega$.  We treat a gain graph $\Phi$ as the biased graph $\langle\Phi\rangle$.  We propose to embed the Rhodes semilattice as an order ideal of a lattice.  The simplest such lattice is the first, which is due to Rhodes and his collaborators.  The fourth is one that is useful to the complexity theory being developed by Rhodes and his collaborators.  The second and third hint at an extension of matroid theory.

\begin{definition}
The \emph{classic Rhodes lattice} of $\Omega$, denoted by $\widehat{R}(\Omega)$, is $R(\Omega)$ with an added top element $\hat{1}$.  $\widehat{R}(\G\pdot K_n)$ is the Rhodes lattice $\widehat{R}_n(\G)$ defined in \cite{Rhodes}.  This lattice was introduced by Rhodes for use in the structure theory of semigroups, where it signifies inconsistency between balanced subgraphs.
\end{definition}

With the benefit of our subgraph interpretation we propose alternative ``Rhodes lattices'' that are more substantial (that is, bigger).

The frame matroid of $\Omega$ is a matroid $\mathbf{F}(\Omega)$ on the edge set $E(\Omega)$ \cite[Section 2]{Zas2}.  
We regard each edge set $S$ as the spanning subgraph $(V(\Omega),S)$. The rank function of $\mathbf{F}(\Omega)$ is $r_\mathbf{F}(S) = n-b(S)$, where $b(S)$ denotes the number of components of $(V(\Omega),S)$ that are balanced.  
We define the \emph{frame flats} of $\Omega$ as the subgraphs $(V(\Omega),S)$ such that $r_\mathbf{F}(S)<r_\mathbf{F}(S \cup e)$ for every $e \in E(\Omega) \setminus S$. (We refer to \cite[Section 2]{Zas2}) for a precise description of frame flats.)

\begin{definition}\label{frameR}
The \emph{frame Rhodes lattice} or \emph{frame subgraph lattice} of $\Omega$, denoted by $R^\mathbf{F}(\Omega)$, is the family of all \emph{frame-closed subgraphs} of $\Omega$, by which we mean the frame flats of the induced subgraphs of $\Omega$.
\end{definition}

The lift matroid of $\Omega$ is another matroid $\mathbf{L}(\Omega)$ on the set $E(\Omega)$ \cite[Section 3]{Zas2}.  
Again we regard each edge set $S$ as a spanning subgraph of $\Omega$, $(V(\Omega),S)$.  The rank function is $r_\mathbf{L}(S) = n-c(S)+\delta(S)$, where $c(S)$ denotes the total number of components of $(V(\Omega),S)$ and $\delta(S) = 0$ if $S$ is balanced, $1$ if $S$ is unbalanced.
We define the \emph{lift flats} of $\Omega$ to be the spanning subgraphs of $\Omega$ whose edge sets are closed in $\mathbf{L}(\Omega)$, which means that $r_\mathbf{L}(S \cup e) > r_\mathbf{L}(S)$ if $e \notin S$.  

\begin{definition}\label{liftR}
The \emph{lift Rhodes lattice} or \emph{lift subgraph lattice}, denoted by $R^{\mathbf{L}}(\Omega)$, is the family of all \emph{lift-closed subgraphs} of $\Omega$, which means all lift flats of all induced subgraphs of $\Omega$.
\end{definition}

Definition \ref{liftR} extends to the natural one-point extension of the lift matroid, the \emph{complete} (or \emph{extended}) \emph{lift matroid}, whose balanced flats are the same as those of $\mathbf{L}(\Omega)$.  We omit the formal definition; see \cite[Section 4]{Zas2}.

Probably the most interesting suggestion for a substantial ``Rhodes lattice'' is the lattice of semiclosed subgraphs.  A subgraph $\Upsilon$ is called \emph{semiclosed}\footnote{Previously called ``balance-closed'' in \cite{Zas1, Zas2}.} if, whenever there is a balanced circle $C\subseteq \Omega$ such that $C \setminus e \subseteq \Upsilon $, then $C \subseteq \Upsilon$.  This seems especially relevant to the Rhodes theory because, if $\G\pdot K_n$ is viewed as a groupoid, then its subgroupoids, which are relevant to Rhodes' complexity theory, are the semiclosed subgraphs.

\begin{definition}\label{semiclosed}
The \emph{semiclosed Rhodes lattice} or \emph{semiclosed subgraph lattice}, denoted by $R^{\mathbf{S}}(\Omega)$, is the family of semiclosed subgraphs of all induced subgraphs of $\Omega$, ordered by inclusion.
\end{definition}

The graphic Rhodes semilattice is the order ideal of balanced subgraphs in each of these proposed Rhodes lattices.

\begin{example}[The different lattices]
The four proposed lattices are all different (unless $\Omega$ is balanced and in a few other cases).  To see this we give a small gain-graphic example: $\Phi = \bbZ_6\pdot K_4$.  The balanced elements are the subgraphs belonging to $R_4(\bbZ_6)$; they are the subgraphs of $K_4$ with any balanced gains.  The only other element of the classic Rhodes lattice $\widehat R(\Phi) = \widehat R_4(\bbZ_6)$ is $\hat1$, which we may regard as $\Phi$ itself since $\Phi$ is unbalanced.  Thus, $\Phi$ is an element of all four lattices.  

In the frame Rhodes lattice $R^\mathbf{F}(\Phi)$ the other lattice elements are the following:
\begin{gather*}
\Phi{:}X_2,\ \Phi{:}X_2 \cup Y_1,\ \Phi{:}X_2 \cup Y_2,\ \Phi{:}X_2 \cup e_{Y_2}, \\
\Phi{:}X_3,\ \Phi{:}X_3 \cup Y_1,
\end{gather*}
where $X_i$ denotes a set of $i$ vertices, $Y_j$ denotes a disjoint set of $j$ vertices (where $i+j\leq4$), and $e_{Y_2}$ denotes any single edge in $\Phi{:}Y_2$.  Each $\Phi{:}X_3$ is the $\bbZ_6$-expansion of $K_3$ with vertex set $X_3$.

In the lift Rhodes lattice $R^\mathbf{L}(\Phi)$ the other lattice elements are the following:
\begin{gather*}
\Phi{:}X_2,\ \Phi{:}X_2 \cup Y_1,\ \Phi{:}X_2 \cup Y_2,\ \Phi{:}X_2 \cup \Phi{:}Y_2, \\
\Phi{:}X_3,\ \Phi{:}X_3 \cup Y_1.
\end{gather*}

The elements of the semiclosed Rhodes lattice $R^\mathbf{S}(\Phi)$ (cf.\ \cite{Gottstein}) are the semiclosed subgraphs of $\Phi$, which are all subgraphs of the forms
\begin{equation*}\begin{gathered}
X_i \text{ for } i=0,1,2,3,4,\\
A{:}X_2,\ A{:}X_2 \cup Y_1,\ A{:}X_2 \cup Y_2,\ A{:}X_2 \cup B{:}Y_2,\\
\fH\pdot K_4{:}X_3,\ \fH\pdot K_4{:}X_3 \cup Y_1,\ \fH\pdot K_4,
\end{gathered}\end{equation*}
and all subgraphs obtained from the last three by switching, where $A{:}X_2$ means a subgraph of $\Phi{:}X_2$ with any nonempty edge set and similarly for $B{:}Y_2$, and $\fH$ is any subgroup of $\bbZ_6$.  (Note that on $K_3$ and $K_4$ we get any group expansion by a subgroup of $\bbZ_6$, not merely the trivial expansion and the $\bbZ_6$-expansion.)  
Among these subgraphs are all those in the other three Rhodes lattice candidates as well as a great many more.  The Rhodes semilattice of $\Phi$ consists of the semiclosed subgraphs that use only the trivial group and have no multiple edges.
\end{example}

\begin{example}[The different lattices on an incomplete graph]
To show what happens with an incomplete group expansion we present a similar example based on the graph $C_4$, the circle of length 4; this is $\Psi = \bbZ_6\pdot C_4$.  Since $C_4$ is a proper subgraph of $K_4$ we get some of the same elements of each Rhodes lattice and some new ones.  

The balanced elements of each Rhodes lattice are the proper subgraphs of $C_4$ with any gain on each edge as well as $C_4$ itself with any balanced gains.  The unbalanced elements are $\Psi$ itself and the following subgraphs:

For $R^\mathbf{F}(\Psi)$:
\begin{gather*}
\Psi{:}X_2,\ \Psi{:}X_2 \cup Y_1,\ \Psi{:}X_2 \cup Y_2,\ \Psi{:}X_2 \cup e_{Y_2}, \\
\Psi{:}X_3=\bbZ_6\pdot P_2,\ \bbZ_6\pdot P_2 \cup Y_1,
\end{gather*}
where now $X_2$ is the vertex set of an edge of $C_4$ and $Y_2$ is the vertex set of the opposite edge, while $P_2$ denotes a path of length 2 in $C_4$ and $Y_1$ is the set of the one vertex not in the $P_2$.

For $R^\mathbf{L}(\Psi)$:  
\begin{gather*}
\Psi{:}X_2,\ \Psi{:}X_2 \cup Y_1,\ \Psi{:}X_2 \cup Y_2,\ \Psi{:}X_2 \cup e_{Y_2}, \\
\Psi{:}X_3=\bbZ_6\pdot P_2,\ \bbZ_6\pdot P_2 \cup Y_1, \Psi{:}X_2 \cup \Psi{:}Y_2, 
\end{gather*}

For $R^\mathbf{S}(\Psi)$:  
\begin{gather*}
A{:}X_2,\ A{:}X_2 \cup Y_1,\ A{:}X_2 \cup Y_2,\ A{:}X_2 \cup B{:}Y_2,\ A{:}X_2 \cup B{:}X'_2, \\
\text{and all switchings of } \fH\pdot C_4
\end{gather*}
where $X_2'$ is the vertex set of an edge adjacent to the edge of $X_2$.  Note that the subgraphs are balanced if $|A|=1$ or $|A|=|B|=1$ as appropriate.
\end{example}

None of the four proposed Rhodes lattices is a geometric lattice (with few, uninteresting exceptions).  
However, the frame and lift Rhodes lattices are combinations of geometric lattices, since the subgraphs with a fixed vertex set do constitute a geometric lattice (cf.\ the remark after Definition \ref{rhodesbg}).  The way this combination works is a topic of current research in order to understand the abstract structure of the original Rhodes semilattice and its generalizations.  

The semiclosed subgraph lattice is essentially the same as the lattice of all subgroupoids of a connected groupoid.  It has been used in the Rhodes theory of semigroup complexity.  The structure of this lattice is currently under investigation from the gain-graph viewpoint.

\section{Acknowledgement}

We are eternally grateful to Stuart Margolis for his generous assistance and advice in understanding the context and literature around the Rhodes semilattice.

\end{document}